\documentclass[12pt, reqno]{amsart}
\usepackage{times, amsmath, amsthm, amscd, amsfonts, amssymb, graphicx, color}
\usepackage[all]{xy}

\newtheorem{theorem}{Theorem}[section]
\newtheorem{lemma}[theorem]{Lemma}
\newtheorem{proposition}[theorem]{Proposition}
\newtheorem{corollary}[theorem]{Corollary}
\theoremstyle{definition}

\theoremstyle{remark}

\numberwithin{equation}{section}

\begin{document}

\title[Big Hankel operators on the Hardy sapce]
{Big Hankel operators on Hardy spaces of strongly pseudoconvex domains}
\author[B.-Y. Chen]
{Bo-Yong Chen}
\address{
Bo-Yong Chen\\
School of Mathematical Sciences, Fudan University, Shanghai, 200433, China}
\email{boychen@fudan.edu.cn}

\author[L. Jiang]
{Liangying Jiang}
\address{
Liangying Jiang \\ Department of Statistics and  Mathematics,
Shanghai Lixin  University of Accounting and Finance,
Shanghai 201209, China}
\email{liangying1231@163.com}

\keywords{Hankel operator, Hardy space,  strongly pseudoconvex domain}

\begin{abstract}
In this article, we investigate the (big) Hankel operators $H_f$ on  Hardy spaces of strongly pseudoconvex domains with smooth boundaries in $\mathbb{C}^n$.  We also  give a necessary and sufficient condition for  boundedness of the Hankel operator $H_f$ on the Hardy space of the unit disc, which is new in the setting of one variable.
\end{abstract}

\maketitle

\section{Introduction}\label{Sec1}

Let $\Omega=\{z\in \mathbb{C}^n: \rho(z)<0\}$ be a strongly pseudoconvex domain with smooth boundary, where $\rho(z)\in C^\infty(\mathbb{C}^n)$ is a strictly plurisubharmonic defining function of $\Omega$, with $d\rho\ne 0$ in some neighborhood of $\partial \Omega$.

Let $d\sigma$ denote the surface measure on  $\partial \Omega$ induced by the Lebesgue measure of $\mathbb C^n$. For $0<p<\infty$, we denote $L^p(\partial\Omega)$ the usual $L^p$  space on  $\partial \Omega$ with respect to the measure $d\sigma$. Let $H^p(\Omega)$ denote the Hardy space of holomorphic functions on $\Omega$, whose norm is given by
$$\|f\|^p_{H^p}:=\sup\limits_{\varepsilon>0}\int_{\partial \Omega_\varepsilon}|f|^pd\sigma_\varepsilon,$$
where $\Omega_\varepsilon=\{z\in \mathbb{C}^n: \rho(z)<-\varepsilon\}$ and $d\sigma_\varepsilon$ is the surface measure on $\partial \Omega_\varepsilon$.
It is well-know  that any $f\in H^p(\Omega)$ has a radial limit at almost all points on $\partial \Omega$ (cf. \cite{Ste72}). Thus one can identify $H^p(\Omega)$ as a closed subspace of  $L^p(\partial\Omega)$. Let $H^\infty(\Omega)$ denote the space of bounded holomorphic functions in $\Omega$.

Let $P: L^2(\partial\Omega)\rightarrow H^2(\Omega)$ be the orthogonal projection via the Szeg\"{o} kernel $S(z, w)$, that is $$Pg(z)= \int_{\partial \Omega}S(z, \zeta) g(\zeta)d\sigma(\zeta)$$
for $g\in L^2(\partial\Omega)$. We may regard this operator as a singular integral operator on $\partial \Omega$, for   $S(z, w)$ is $C^\infty$ on $\partial\Omega\times \partial\Omega\setminus\Delta$ (where $\Delta$ is the diagonal of $\partial\Omega\times \partial\Omega$). We also define the Poisson-Szeg\"{o} kernel by $P(z, w)=S(z,z)^{-1}|S(z,w)|^2$ for $z\in \Omega$ and $w\in \partial \Omega$ (see \cite{Ste72}).

For $f\in L^2(\partial\Omega)$, the (big) Hankel operator $H_f$ with symbol $f$ is the operator from $H^2(\Omega)$ into $(H^2(\Omega))^\bot$ defined by
$$H_f g =(I-P)(fg), \ \ g\in H^2(\Omega).$$
We may identify $H_f$ with the operator $(I-P)M_f P$ on the Hilbert space $L^2(\partial\Omega)$, where $M_f$ is the multiplication operator given by $(M_f g)(z)=f(z)g(z)$, $z\in \overline{\Omega}$. If $f$ is bounded on $\partial\Omega$, then clearly $H_f$ is bounded with $\|H_f\|\le \|f\|_\infty$. In general, $H_f$ is only densely defined, whose domain contains $H^\infty(\Omega)$.

This operator is a generalization of the Hankel operator in the case of the unit  disc $D$, which is different from the (small) Hankel operator $h_f$ defined by $h_f g=\overline{P(fg)}:H^2(\Omega)\rightarrow \overline{H_0^2(\Omega)}$, where $\overline{H_0^2(\Omega)}$ is the space of complex conjugate of functions in $H^2(\Omega)$ which are zero at the origin of $\mathbb{C}^n$. The theory of Hankel operators of the unit disc is classical. There are several books dealing with Hankel operators on the Hardy space $H^2(D)$, see \cite{Par88}, \cite{Pel03}, \cite{Pow82} and \cite{Zhu07} for details.

For the unit disc, the characterization of bounded Hankel operators on the Hardy space is due to Nehari \cite{Neh57}. The compactness is given by Hartman  \cite{Har58}. Schatten class Hankel operators  is characterized  by Peller \cite{Pel82}, \cite{Pel85} and Semmes \cite{Sem84}.

The extension of these results to  higher dimensional domains for (big) Hankel operators is proved to be resistant in the case of the unit ball $B_n$. A fundamental result due to Coifman, Rochberg and Weiss \cite{CRW76}  asserts that $[M_f, P]=M_f P-PM_f$ is bounded on $L^2(\partial B_n)$ for $f\in BMO$, and $[M_f, P]$ is compact if $f\in VMO$. The commutator result implies that $H_f=[M_f, P] P$ is bounded on $L^2(\partial B_n)$ for every $f\in BMO$ and $H_f$ is compact if $f\in VMO$. In fact, by virtue of the relation $$[M_f, P]=H_f-H_{\overline{f}}^*,$$ the study of the commutator $[M_f, P]$ is equivalent to the problem concerning both $H_f$ and $H_{\overline{f}}$.

The study of $H_f$ alone is usually more difficult than  corresponding problems concerning $H_f$ and $H_{\overline{f}}$ simultaneously. Since $H_f=H_{f-Pf}$, properties of the  Hankel operator $H_f$ for general $f\in L^2(\partial\Omega)$ are recovered by  properties of $f-Pf$.

For $f\in L^2(\partial B_n)$, it was showed in \cite{Xia08} that $H_f$ is bounded on the Hardy space $H^2(B_n)$ if and only if $f-Pf\in BMO$, and $H_f$ is compact if and only if $f-Pf\in VMO$. Moreover, in \cite{FX09}, the authors proved that $H_f$ belongs to the Schatten class $S_p$, $2n<p<\infty$, if and only if $f$ lies  in the Besov space $B_p$. Other characterizations for  the Hankel operator $H_f$ on the Hardy space $H^2(B_n)$  were given in  \cite{FR90}, \cite{Xia04} and \cite{Zhe97}.

In particular,  Amar   \cite{Amar95} established links between the Hankel operator $H_f$ on the Hardy space  $H^2(B_n)$ and  the $\overline{\partial}$ of a Stokes extension of $f$ in $B_n$, which is also the motivation of our work. In addition, we mention the work of  Li and Luecking (see [20-22]  and \cite{Lue92}), where  the theory of  $\overline{\partial}$-operator is used to give  complete characterizations of  boundedness, compactness and Schatten ideals of Hankel operators acting  on  Bergman spaces of strongly pseudoconvex domains.

In this article, we try to give some characterizations for  boundedness of  Hankel operators on the Hardy space $H^2(\Omega)$, based on the theory of $BMO$, $VMO$, and  $\overline{\partial}$-operator,
when $\Omega$ is a bounded strongly pseudoconvex domain  in $\mathbb{C}^n$. Our main result is the following.

\begin{theorem}\label{thm 1.1}
Let  $\Omega$ be a bounded  strongly pseudoconvex domain with smooth boundary and $f\in L^2(\partial\Omega)$. If $f\in BMO(\partial\Omega)$, then the Hankel operator $H_f$ is bounded from $H^2(\Omega)$ into $L^2(\partial\Omega)$.
\end{theorem}

The idea  for proving this result comes from Amar's $\overline{\partial}$-method and the characterization of $BMO$ given by Varopoulos \cite{Var77}. This method is different from  the proofs of corresponding results in the setting of the unit disc and the unit ball. Particularly, in the proof of Theorem 1.1, we get   a solution to  Gleason's  problem on the Hardy space $H^p(\Omega)$ for $1\le p\le\infty$.

In fact, the proof of Theorem 1.1 shows that $H_f$ is bounded on $H^2(\Omega)$ if  $f$  admits a decomposition $f=f_1+f_2$ with $f_1\in L^\infty(\partial\Omega)$ and $f_2$  having  an extension $\widetilde{f_2}$  in $\Omega$, such that
 $|\overline{\partial}\widetilde{f_2}|dV$ and $\frac{1}{\sqrt{-\rho}}|\overline{\partial}\widetilde{f_2}\wedge\overline{\partial}\rho|dV$ are  Carleson measures in $\Omega$, where $dV$ is the Lebesgue measure on  $\Omega$. We conjecture that this condition is also necessary.

Based on this observation, we  give  the following $\overline{\partial}$-characterization  for Hankel operators on the Hardy space $H^2(D)$ of the unit disc.

\begin{theorem}\label{thm 1.2} If $f\in L^2(\partial D)$, then the following properties are equivalent.
\\ (1) $H_f$ is bounded from $H^2(D)$ into $L^2(\partial D)$.
\\ (2) The function $$F(z)=\inf\limits_{h\in H^2(D)}\frac{1}{2\pi}\int|f-h|^2P_z(\theta)d\theta \eqno(1.1) $$

is bounded.
\\ (3) The measure $|\frac{\partial f}{\partial \overline{z}}|^2\log\frac{1}{|z|}dxdy$ is a Carleson measure, where $f(z)=\int f(\theta)P_z(\theta)d\theta$

is the Poisson integral of $f$.
\end{theorem}

Moreover, applying Berndtsson's $L^2$-estimations for  solutions of $\overline{\partial}_b$-equation, we shall answer the question about the boundedness of Hankel operators on the Hardy space $H^2(\Omega)$, when $\Omega$ is a general pseudoconvex domain.  We shall also use Donnelly-Fefferman's $L^2$-estimates  to give a new proof for the $\overline{\partial}$-characterization of  $H_f$ on the Bergman space $A^2(\Omega)$. This result  has been proved in Li \cite{Li94}, by using the integral representation of a solution to $\overline{\partial}$-equation.

This paper is divided into six sections. After some necessary preliminaries in Section 2, we   solve the Gleason's  problem on the Hardy space $H^p(\Omega)$. In  Section 3, we   give the proof of Theorem 1.1 and generalize this result to  Hardy spaces $H^p(\Omega)$ for $p>1$.  The compactness of these Hankel operators shall be discussed in this section. In Section 4, we try to  investigate the   boundedness of Hankel operators when acting on Hardy spaces of general  pseudoconvex domains. Similarly,  we  use a different method to prove the corresponding result on the Bergman space of strongly pseudoconvex domain. Section 5 devotes to give a new characterization for  Hankel operator on the Hardy space in the case of the unit disc. In last section, we show that both $H_f$ and $H_{\overline{f}}$ are bounded on the Hardy space $H^2(\Omega)$ if and only if $f\in BMO(\partial\Omega)$, when $\Omega$ is a bounded strongly pseudoconvex domain, and $1/S(\cdot, z)$ is holomorphic and bounded in $\Omega$ for fixed $z\in \Omega$.

\section {Geometry of strongly pseudoconvex domains}\label{Sec2}

The geometry of strongly pseudoconvex domains is  well understood (see \cite{Hom73} and \cite{Ran86}). In this section, we collect some notations and results needed later.

Let $d(p, q)$ be the Koranyi distance between points on $\partial \Omega$ and $B(p, r)=\{\zeta\in \partial \Omega:  d(\zeta, p)<r\}$ the corresponding balls. The tent $Q(p, r)$ over $B(p, r)$ is $$Q(p, r)=\{z\in \Omega: d(z', p)+d(z)<r\},$$ where $d(z)=d(z, \partial\Omega)$ is the Euclidean distance and $z'$ is the projection of $z$ on $\partial\Omega$. A positive measure $d\mu$ on $\Omega$ is a Carleson measure if
$$\|\mu\|=\sup\limits_{p\in \partial \Omega, r>0}\frac{1}{\sigma(B(p, r))}\int_{Q(p, r)}d\mu=C$$
is finite. It is well-known \cite{Hom67} that $d\mu$ is a Carleson measure if and only if
$$\int_\Omega|\phi|^pd\mu\le C\|\phi\|^p_{H^p}, \ \ \ \  \phi\in H^p(\Omega).\eqno(2.1)$$
Also, we call $d\mu$ a vanishing Carleson measure if $$\lim\limits_{j\to \infty}\int_\Omega|\phi_j|^pd\mu=0$$
for any sequence $\{\phi_j\}$, which is bounded in $H^p(\Omega)$ and converges to $0$ uniformly on any compact subset of $\Omega$ as $j\to\infty$.

For $0\le \alpha\le 1$, we denote by  $W^\alpha$ the interpolation spaces between the space  $W^1$ of Carleson measures and the space $W^0$ of finite measures  in $\Omega$. If $0<\alpha<1$, then the measure $d\mu$ is in $W^\alpha$ if and only if $d\mu=\phi d\tau$, where $d\tau\in W^1$, $\phi\in L^p(d\tau)$ and $1/p=1-\alpha$ (see \cite{AB79}). In view of (2.1), if $\phi\in H^p(\Omega)$ and $d\tau\in W^1$, then $\phi d\tau\in W^\alpha$ and
$$\|\phi d\tau\|_{W^\alpha}\lesssim \|\phi\|_{H^p} \|d\tau\|_{W^1}. \eqno(2.2)$$

Let $\omega$ be a $\overline{\partial}$-closed $(0,1)$ form in $\Omega$. If $u$ is a function defined on $\partial \Omega$ that can be obtained as the boundary value of a solution to $\overline{\partial} u=\omega$,
then we write$$\overline{\partial}_b u=\omega$$
 (see \cite{Sko76}, \cite{AB79} or \cite{Ber87}).
Following Skoda \cite{Sko76}, this equation can be given by $$\int_{\partial\Omega}uh=\int_\Omega \omega\wedge h.   \eqno(2.3)$$
for any smooth $\overline{\partial}$-closed form $h$ of bidegree $(n, n-1)$.

In \cite{AB79}, Amar and Bonami gave the following  on the smoothly bounded strongly pseudoconvex domains, which is very important for proving our main result.
\medskip
\\
{\bf Theorem A.}  {\it Suppose that $\omega$ is a smooth $\overline{\partial}$-closed $(0,1)$ form on $\overline{\Omega}$  and $p>1$. If the coefficients of $\omega$ and $\frac{\overline{\partial} \rho\wedge \omega}{\sqrt{-\rho}}$ belong to the space $W^\alpha$   with $1/p=1-\alpha$,  then there is a solution of $\overline{\partial}_b u=\omega$ continuous on $\overline{\Omega}$ such that
 $$\|u\|_{L^p(\partial \Omega)}\le C\biggl\|\biggl(|\omega|+\frac{|\omega\wedge \overline{\partial}\rho|}{\sqrt{-\rho}}\biggr)dV\biggr\|_{W^\alpha} $$ for a constant $C>0$.}

\medskip

The space $BMO(\partial\Omega)$ is defined in terms of the Koranyi distance as the space of functions such that
$$\|f\|_{BMO}:=\sup\limits_{p\in \partial\Omega, r>0}\frac{1}{\sigma(B(p, r))}\int_{B(p, r)}|f(\zeta)-f_{B(p, r)}|d\sigma(\zeta)<\infty,$$
where $f_{B(p, r)}$ is the average of $f$ over the ball $B(p, r)$. Moreover, $VMO(\partial\Omega)$ is defined as the subspace of $BMO$ which is the closure of the continuous functions in the $BMO$ topology. We also put $BMOA=BMO\cap\mathcal{O}(\Omega)$ and $VMOA=VMO\cap\mathcal{O}(\Omega)$

When $\Omega$ is strongly pseudoconvex, the reproducing property of the Poisson-Szeg\"{o} kernel gives
 $$
f(z)=\int_{\partial\Omega}f(\zeta)P(z, \zeta)d\sigma(\zeta) \ \ \ \ \  \mbox{for} \ f\in H^2(\Omega)
$$
(see \cite{Ste72}).
It follows from  \cite{LL07} that  $f\in BMOA(\Omega)$  if and only if
$$\|f\|^2_{BMOA}=\sup\biggl\{\int_{\partial\Omega}|f(\zeta)-f(z)|^2P(z, \zeta)d\sigma(\zeta); z\in \Omega\biggr\}<\infty,$$
by using similar arguments as in the book of Garnett \cite{Gar81} and the John-Nireberg Lemma. In fact, as  in \cite{Gar81}, one may show that  these characterizations also hold for $BMO(\partial\Omega)$ and  $VMO(\partial\Omega)$, the only difference is that $f(z)$ should be replaced by  $$\widetilde{f}(z)=\int_{\partial\Omega} f(\zeta)  P(z, \zeta)d\sigma(\zeta),$$  the Poisson-Szeg\"{o} integral of $f$.

By \cite{For76} there is a smooth function $\Phi(\zeta, z)$ on $\overline{\Omega}\times\overline{\Omega}$ which is holomorphic in $z$ for fixed $\zeta\in \overline{\Omega}$ and satisfies
$$2\mbox{Re} \Phi\ge -\rho(\zeta)-\rho(z)+\delta|\zeta-z|^2$$ and $d_\zeta \overline{\Phi}|_{\zeta=z}=-d_z\Phi|_{\zeta=z}=-\partial\rho(\zeta)$.
It is well-known that $|\Phi(\zeta, z)|$ is compatible with the Koranyi distance $d(\zeta, z)$. Moreover, there is a form $Q(\zeta, z)=\sum\limits_{j=1}^n Q_j(\zeta,z)d\zeta_j$, holomorphic in $z$, such that $\Phi(\zeta, z)=\sum\limits_{j=1}^n Q_j(\zeta, z)(z_j-\zeta_j)$ and  the following  Cauchy-Fantappi\`{e} formula holds
$$Fu(z)=c\int_{\partial \Omega}u(\zeta)\frac{Q\wedge(\overline{\partial}Q)^{n-1}}{\Phi(\zeta, z)^n}d\sigma(\zeta), \ \ \ z\in\Omega.$$
Clearly, $Fu$ is holomorphic in $\Omega$ if $u\in L^p(\partial\Omega)$, $1\le p\le \infty$ and $Fu=u$ if $u$ is the boundary value of a holomorphic function. In fact, the kernel in the Cauchy-Fantappi\`{e} formula is also called the Henkin-Ramirez reproducing kernel. Thus, $Fu$ has admissible boundary values a.e. if  $u\in L^p(\partial\Omega)$,  $p>1$, and this operator maps $L^p(\partial\Omega)$ into $H^p(\Omega)$, $BMO$ into $BMOA$, we refer to  \cite{Hen70} or \cite{Ran86} for details. We also notice that the Cauchy-Fantappi\`{e} formula for a holomorphic function $u$ can be written as follows
$$u(z)=\int_{\partial \Omega}\frac{A(\zeta, z)u(\zeta)}{\Phi(\zeta, z)^n}d\sigma(\zeta), $$
where $A(\zeta, z)$ is smooth and holomorphic in $z$.

To conclude this section, we will  solve the following Gleason's problem in the setting  of  bounded strongly pseudoconvex domains. Let $X$ be some class of holomorphic functions in a region $\Omega\subset \mathbb{C}^n$. Gleason's problem for $X$ is the following:

{\it If $a\in \Omega$ and $f\in X$, do there exist functions $g_1, \ldots, g_n\in X$ such that $$f(z)-f(a)=\sum\limits_{j=1}^n(z_j-a_j)g_j(z)$$
 for all $z\in \Omega$?}

Gleason \cite{Gle64} originally asked this problem for $\Omega=B_n$, $a=0$ and $X=A(B_n)$, the ball algebra, which was solved by Leibenson. If  $X$ is one of the spaces $H^p(B_n)$ ($1\le p\le\infty$), $A(B_n)$ and $\mbox{Lip}_\alpha$ ($0<\alpha\le 1$), then the  problem was also solved (see  \cite{Rud80}). The simplest solution is probably the one found by Ahern and Schneider \cite{AS95}, which also works for strongly pseudoconvex domains when $X=\mbox{Lip}_\alpha$  or  $A(\Omega)$. Here, we shall solve the problem for  $X=H^p(\Omega)$ when $\Omega$ is strongly pseudoconvex.

First of all, we recall the following two results due to Stout \cite{Stou76}.
\medskip
\\
{\bf Theorem B.} {\it If $f\in H^p(\Omega)$, $1\le p\le\infty$ and if $\gamma$ is a  Lipschitz continuous function on $\mathbb{C}^n$, then the function $$F(z)=\int_{\partial \Omega}f(\zeta)\gamma(\zeta)\frac{Q\wedge(\overline{\partial}Q)^{n-1}}{\Phi(\zeta, z)^n}d\sigma(\zeta)$$
belongs to $H^p(\Omega)$.}

\medskip

Note that
$$F(z)=\int_{\partial \Omega}f(\zeta)[\gamma(\zeta)-\gamma(z)]\frac{Q\wedge(\overline{\partial}Q)^{n-1}}{\Phi(\zeta, z)^n}d\sigma(\zeta)+\gamma(z)f(z).$$
Thus,  in order to proving  Theorem B, the author  considered   the operator $T=T_\gamma$ defined by
$$Tg(z)=\int_{\partial \Omega}g(\zeta)[\gamma(\zeta)-\gamma(z)]\frac{Q\wedge(\overline{\partial}Q)^{n-1}}{\Phi(\zeta, z)^n}d\sigma(\zeta),$$
and obtained  the following boundedness properties of the operator $T$.
\medskip
\\
{\bf Theorem C.}\  \ \  (1) {\it If $g\in L^q(\partial\Omega)$, $q\ge 2n$, then $Tg$ is bound on $\Omega$.}
\ \ (2) {\it If $g\in L^p(\partial\Omega)$, $1\le p<\infty$, then
 $$\sup\limits_{r<0}\int_{\rho=r}|Tg(\zeta)|^pd\sigma(\zeta)<\infty.$$}

Now we  may solve  Gleason's  problem  for $H^p(\Omega)$.

\begin{proposition}\label{prop 2.1}
Let $\Omega$ be a bounded strongly pseudoconvex domain with smooth boundary. If $f\in H^p(\Omega)$, $1\le p\le \infty$, and $a\in \Omega$, then there exist functions $g_1,\ldots, g_n\in H^p(\Omega)$ such that $$f(z)-f(a)=\sum\limits_{j=1}^n(z_j-a_j)g_j(z).$$
\end{proposition}

\begin{proof}
We follow the arguments used in solving Gleason's problem for the case $X=A(\Omega)$. As in Theorem 4.2 of \cite{Ran86},  the reproducing property of the Henkin-Ramirez kernel gives
{\setlength\arraycolsep{2pt}
\begin{eqnarray*}
f(z)-f(a)&=&\int_{\partial \Omega}f(\zeta)\biggl[\frac{A(\zeta, z)}{\Phi(\zeta, z)^n}-\frac{A(\zeta, a)}{\Phi(\zeta, a)^n}\biggr]d\sigma(\zeta)
\\ &=& \int_{\partial \Omega}f(\zeta)\biggl[\frac{A(\zeta, z)\Phi(\zeta, a)^n-A(\zeta, a)\Phi(\zeta, z)^n}{\Phi(\zeta, z)^n\Phi(\zeta, a)^n}\biggr]d\sigma(\zeta)
\end{eqnarray*}}for $f\in H^p(\Omega)$ and $z, a\in \Omega$.

By  \cite{Ovr71}, p. 148, there are functions $L_j, j=1,\ldots, n$, which are $C^\infty$ on $(\partial\Omega)\times \overline{\Omega}\times \overline{\Omega}$ and  holomorphic in $z$ and $a$, such that
$$A(\zeta, z)\Phi(\zeta, a)^n-A(\zeta, a)\Phi(\zeta, z)^n=\sum\limits_{j=1}^n(z_j-a_j)L_j(\zeta, z, a).$$
It follows that
$$f(z)-f(a)=\sum\limits_{j=1}^n(z_j-a_j)\int_{\partial \Omega}f(\zeta)\frac{L_j(\zeta, z, a)}{\Phi(\zeta, z)^n\Phi(\zeta, a)^n} d\sigma(\zeta).\eqno(2.4)$$
Define $$
T_j f(z, a)=\int_{\partial \Omega}f(\zeta)\frac{L_j(\zeta, z, a)}{\Phi(\zeta, z)^n\Phi(\zeta, a)^n} d\sigma(\zeta).
$$
Clearly, $T_jf$ is holomorphic on $\Omega\times \Omega$. If we show that  $T_j f(z, a)\in H^p(\Omega)$ for fixed  $a$, we were done.

Note that Corollary 3.10 in \cite{Ran86} gives  $A(\zeta, z)=h(\zeta)+O(|\zeta-z|)$, where $h(\zeta)\ne 0$ for all $\zeta\in \partial\Omega$. Set
$$\chi(\zeta, z, a)=\frac{L_j(\zeta, z, a)}{h(\zeta)\Phi(\zeta, a)^n}.$$
We have   $\chi(\zeta, z, a)=\chi(\zeta, \zeta, a)+O(|\zeta-z|)$. A calculation shows that
{\setlength\arraycolsep{2pt}
\begin{eqnarray*}
&&\frac{L_j(\zeta, z, a)}{\Phi(\zeta, z)^n\Phi(\zeta, a)^n}=\frac{L_j(\zeta, z, a)}{h(\zeta)\Phi(\zeta, a)^n}\cdot \frac{h(\zeta)}{\Phi(\zeta, z)^n}
\\ &&= \chi(\zeta, z, a) \frac{A(\zeta, z)}{\Phi(\zeta, z)^n}+\chi(\zeta, z, a)\frac{O(|\zeta-z|)}{\Phi(\zeta, z)^n}
\\ &&= \chi(\zeta, \zeta, a) \frac{A(\zeta, z)}{\Phi(\zeta, z)^n}+O(|\zeta-z|)\frac{A(\zeta, z)}{\Phi(\zeta, z)^n}+\chi(\zeta, z, a)\frac{O(|\zeta-z|)}{\Phi(\zeta, z)^n}.
\end{eqnarray*}}This implies
{\setlength\arraycolsep{2pt}
\begin{eqnarray*}
T_j f(z, a)&&=\int_{\partial \Omega}f(\zeta)\chi(\zeta, \zeta, a)\frac{A(\zeta, z)}{\Phi(\zeta, z)^n}d\sigma(\zeta)
\\ && +\int_{\partial \Omega}f(\zeta)O(|\zeta-z|)\frac{A(\zeta, z)}{\Phi(\zeta, z)^n}d\sigma(\zeta)+ \int_{\partial \Omega}f\chi\frac{O(|\zeta-z|)}{\Phi(\zeta, z)^n}d\sigma(\zeta).\ \ \ \ (2.5)
\end{eqnarray*}}

Next, we will need a technique of  Stout \cite{Stou76} to estimate (2.5), which is different from that in the proof of Theorem 4.2 in \cite{Ran86}. Since $\chi(\zeta, \zeta, a)\in C^\infty$, it follows from Theorem B that
$$F(z)=\int_{\partial \Omega}f(\zeta)\chi(\zeta, \zeta, a)\frac{A(\zeta, z)}{\Phi(\zeta, z)^n}d\sigma(\zeta)$$
belongs to $H^p(\Omega)$, $1\le p\le\infty$.

Since the functions $A$ and  $\chi$ are smooth, we may use similar ideas to estimate  the second term in (2.5). It remains to consider the last term.   Set $$Tf(z)=\int_{\partial \Omega}f\chi\frac{O(|\zeta-z|)}{\Phi(\zeta, z)^n}d\sigma(\zeta)$$
and $T^{(r)}f= Tf|_{\{\rho=r\}}$  when  $r$ is close to $0$. Once we  prove the following  claim
$$\|T^{(r)}f\|_{L^\infty(\{\rho=r\})}\le C\|f\|_{L^\infty(\partial\Omega)}\ \ \mbox{and}\ \ \|T^{(r)}f\|_{L^1(\{\rho=r\})}\le C\|f\|_{L^1(\partial\Omega)}$$
for a constant $C>0$,  the  Riesz-Thorin theorem implies that if $f\in L^p(\partial\Omega)$, $1<p<\infty$, then $$\|T^{(r)}f\|_{L^p(\{\rho=r\})}\le C\|f\|_{L^p(\partial\Omega)}.$$
Since the constant $C$ is independent of $r$, we obtain that $Tf\in L^p(\partial\Omega)$ for $1<p<\infty$.

To verify the claim, we first assume  $f\in L^\infty(\partial\Omega)$.  Since   $\chi$ is  smooth,  we have
$$|Tf(z)|  \lesssim \|f\|_\infty\biggl(\int_{\partial\Omega}\frac{O(|\zeta-z|)}{|\Phi(\zeta, z)|^n}d\sigma(\zeta)\biggr)$$ for any $z\in \Omega$.
In the proof of Theorem C, the author has showed  that the above integral is finite. The proof of Theorem C also  yields that
$$\int_{\rho=r}\frac{O(|\zeta-z|)}{|\Phi(\zeta, z)|^n}d\sigma(z)$$
are bounded uniformly in $\zeta, r$, where $\zeta\in \partial\Omega$ and $r$  is close to $0$. Thus, if $f\in L^1(\partial\Omega)$, we get that
$$\int_{\rho=r}|Tf(z)|d\sigma(z)$$ are bounded uniformly in $r$  near  $0$.  Thus the claim is verified.

Combining the above arguments with (2.5), we find that   $T_jf(\cdot, a)\in H^p(\Omega)$ for $1\le p\le \infty$.  Taking $g_j(z)=T_jf(z, a)$,  we obtain the desired result   from (2.4).
\end{proof}

 Proposition 2.1 gives a  decomposition for  $H^2(\Omega)$. Similar as  the proof of Lemma 2.5 in \cite{Amar95}, we may easily deduce the following lemma.

\begin{lemma}\label{lem 2.2}
The space $H^2(\Omega)^\bot$ can be identified with the space of $\overline{\partial}_b$-closed forms of  bidegree $(n, n-1)$  satisfying
$$ h\in H^2(\Omega)^\bot\rightarrow \mathcal{H}\in L^2_{(n, n-1)}(\partial\Omega), \ \ \ \ \overline{\partial}_b\mathcal{H}=0$$
and $$\int_{\partial\Omega} \phi\cdot \overline{h} d\sigma=\int_{\partial\Omega} \phi\cdot\mathcal{H}, \ \ \ \  \forall\, \phi\in L^2(\partial\Omega).$$
\end{lemma}

As an easy consequence of this lemma and Stokes' theorem,  we obtain the following result which generalizes a result of  Amar \cite{Amar95} in  the unit ball.

\begin{proposition}\label{prop 2.3}
Suppose that $\Omega$ is a bounded strongly pseudoconvex domain with smooth boundary. Let  $f\in L^2(\partial\Omega)$ and $\widetilde{f}$ be any Stokes' extension of $f$ in $\Omega$. Then
$$\overline{\partial}_b(H_f g)=g\cdot \overline{\partial} \widetilde{f}, \ \  \ \ \  \forall\, g\in H^\infty (\Omega).$$
\end{proposition}

\begin{proof}
The argument  is similar to Proposition 2.4 in \cite{Amar95}. For  convenience, we give a simple proof here.

Let $g\in H^\infty (\Omega)$ and $u=H_f g=fg-P(fg)$. Applying Lemma 2.2 and Stokes' theorem,  for any $h\in H^2(\Omega)^\bot$, we get
{\setlength\arraycolsep{2pt}
\begin{eqnarray*}
\int_{\partial\Omega} u\cdot \overline{h} d\sigma &=& \int_{\partial\Omega} [fg-P(fg)]\cdot \overline{h} d\sigma = \int_{\partial\Omega} fg\cdot \overline{h} d\sigma
\\ &=& \int_{\partial\Omega} fg\cdot \mathcal{H}=  \int_{\Omega} g \cdot\overline{\partial}\widetilde{f}\wedge\mathcal{H}.
\end{eqnarray*}}This combined with (2.3) yields $\overline{\partial}_b u=g\cdot \overline{\partial}\widetilde{f}$, which   completes the proof.
\end{proof}

{\bf Remark.} In fact, if $g\in H^\infty (\Omega)$ and  $f\in L^2 (\partial\Omega)$, then $P(fg)\in H^2 (\Omega)$. Now, applying the definition of $\overline{\partial}_b$ in Page 166 of \cite{CS00},  we easily deduce that  $H_f g=fg-P(fg)$ is a solution of $\overline{\partial}_b u=g \cdot\overline{\partial}\widetilde{f}$ for the extension $\widetilde{f}$ of $f$.

\section{Hankel operators on  Hardy spaces of strongly pseudoconvex domains with smooth boundarys}\label{Sec3}

Before proving Theorem 1.1, we shall introduce one more notion. Let $f\in C^1(\Omega)$, we define the nonisotropic gradiant near the boundary of $\Omega$   as
$$|Df|=|\nu_0(f)|+|\mu_0(f)|+|\rho|^{-1/2}\sum\limits_{j=1}^{2n-2}|\mu_j(f)|$$
(see \cite{Var77}), where $\nu_0$ is a normalized vector field in some neighborhood of $\partial\Omega$ that is normal and directed inwards to $\partial\Omega$ at every point $\zeta\in \partial\Omega$, $\mu_0$ is a normalized field that is  orthogonal to the holomorphic tangent space at every point  $\zeta\in \partial\Omega$, and fields $\mu_1,\ldots, \mu_{2n-2}$ are constructed to  form an orthonormal basis at $\zeta$.

In \cite{CRW76}, the authors used the theory of singular integral operators to  show that $H_f$ is bounded on the Hardy space $H^2(B_n)$ if $f\in BMO(\partial B_n)$. Amar \cite{Amar95} applied the theory of $\overline{\partial}$ to prove that $H_f$ is bounded on $H^2(B_n)$ if $f$ admits a Stokes' extension $\widetilde{f}$  to $B_n$  such that both $|\overline{\partial}\widetilde{f}|dV$ and $\frac{1}{\sqrt{-\rho}}|\overline{\partial}\widetilde{f}\wedge\overline{\partial}\rho|dV$ are Carleson measures. Motivated by this result and the work of Varopoulos \cite{Var77} about the $BMO$ theory, we may generalize the result of \cite{CRW76}  to the bounded strongly pseudoconvex domains.

\begin{proof}[Proof of Theorem 1.1]
Let $f\in BMO(\partial\Omega)$. By Theorem 2.1.1 of \cite{Var77},  $f$ admits a decomposition $f=f_1+f_2$, such that $f_1\in L^\infty(\partial\Omega)$ and $f_2$ has an extension $\widetilde{f_2}$ in $\Omega$, which  satisfies that
$\lim\limits_{\lambda\rightarrow 0}\widetilde{f_2}(\zeta+\lambda \nu_0(\zeta))=\widetilde{f_2}(\zeta)$, $\forall\, \zeta\in \partial\Omega$,
and $|D\widetilde{f_2}|dV$ is a Carleson measure.

Since $f_1\in L^\infty(\partial\Omega)$, it is easy to see that the multiplication operator  $M_{f_1}$ is bounded from $H^2(\Omega)$ into $L^2(\partial\Omega)$,  which
 implies that $H_{f_1}=(I-P)M_{f_1}$ is bounded. Thus, it suffices  to prove the boundedness of $H_{f_2}$.

Now,  $f\in L^2(\partial\Omega)$ and $f_1\in L^\infty(\partial\Omega)$ give that $f_2\in L^2(\partial\Omega)$.  Since $|D \widetilde{f_2}|dV$ is a Carleson measure, we infer from Lemma 2.5.1 in \cite{Var77} that both $|\overline{\partial}\widetilde{f_2}|dV$ and $\frac{1}{\sqrt{-\rho}}|\overline{\partial}\widetilde{f_2}\wedge\overline{\partial}\rho|dV$ are  Carleson measures in $\Omega$.
 Let $g\in H^2(\Omega)$. Applying (2.2) to $|g\cdot \overline{\partial} \widetilde{f_2}|dV$ and $\frac{1}{\sqrt{-\rho}}|g\cdot \overline{\partial}\widetilde{f_2}\wedge\overline{\partial}\rho|dV$, we deduce that the coefficients of  $g\cdot \overline{\partial} \widetilde{f_2}$ and $g\cdot \overline{\partial}\widetilde{f_2}\wedge\overline{\partial}\rho/\sqrt{-\rho}$  belong to the space $W^{1/2}$.   Thus there is, according to  Theorem A,  a solution of the equation $\overline{\partial}_b u=g\cdot \overline{\partial}\widetilde{f_2}$  which satisfies $u\in L^2(\partial\Omega)$. Moreover,  using  Theorem A and (2.2), we obtain
{\setlength\arraycolsep{2pt}
\begin{eqnarray*}\|u\|_{L^2(\partial\Omega)}&\lesssim & \| |g\cdot \overline{\partial} \widetilde{f_2}|dV\|_{W^{1/2}}+\biggl\| \frac{1}{\sqrt{-\rho}}|g\cdot \overline{\partial}\widetilde{f_2}\wedge\overline{\partial}\rho|dV\biggr\|_{W^{1/2}}
\\ &\lesssim & \|g\|_{H^2}\biggl(\| |\overline{\partial} \widetilde{f_2}|dV\|_{W^1}+\biggl\|\frac{1}{\sqrt{-\rho}}| \overline{\partial}\widetilde{f_2}\wedge\overline{\partial}\rho|dV\biggr\|_{W^1}\biggr).
\end{eqnarray*}}

In view of Proposition 2.3,  we have $\overline{\partial}_b( H_{f_2}g)=g\cdot \overline{\partial}\widetilde{f_2}$ when $g\in H^\infty(\partial\Omega)$. Using the uniqueness of the solution orthogonal to $H^2(\Omega)$, we get
{\setlength\arraycolsep{2pt}
\begin{eqnarray*}\|H_{f_2}g\|_{L^2(\partial\Omega)}&\le & \|u\|_{L^2(\partial\Omega)}
\\ &\lesssim& \|g\|_{H^2}\biggl(\| |\overline{\partial} \widetilde{f_2}|dV\|_{W^1}+\biggl\|\frac{1}{\sqrt{-\rho}}| \overline{\partial}\widetilde{f_2}\wedge\overline{\partial}\rho|dV\biggr\|_{W^1}\biggr).
\end{eqnarray*}}Since $H^\infty(\partial\Omega)$ is dense in $H^2(\Omega)$, we see that $H_{f_2}$ is bounded on $H^2(\Omega)$ and
$$\|H_{f_2}\| \lesssim \biggl(\| |\overline{\partial} \widetilde{f_2}|dV\|_{W^1}+\biggl\|\frac{1}{\sqrt{-\rho}}| \overline{\partial}\widetilde{f_2}\wedge\overline{\partial}\rho|dV\biggr\|_{W^1}\biggr).$$
 It follows that $H_f=H_{f_1}+H_{f_2}$  is bounded from $H^2(\Omega)$ into  $L^2(\partial\Omega)$.
\end{proof}

From  the proof of Theorem 1.1, we may also give the following characterization for the compactness of $H_f$ on the Hardy space $H^2(\Omega)$.

\begin{theorem}\label{thm 3.1}
Let  $f\in L^2(\partial\Omega)$. If $f$  admits a decomposition $f=f_1+f_2$ with $f_1\in C(\partial\Omega)$ and $\widetilde{f_2}$  being an extension of $f_2$  in $\Omega$, such that
  $|\overline{\partial}\widetilde{f_2}|dV$ and $\frac{1}{\sqrt{-\rho}}|\overline{\partial}\widetilde{f_2}\wedge\overline{\partial}\rho|dV$ are vanishing Carleson measures in $\Omega$,
then the Hankel operator $H_f$ is compact from $H^2(\Omega)$ into $L^2(\partial\Omega)$.
\end{theorem}

\begin{proof}
For $f_1\in C(\partial\Omega)\subset L^\infty(\partial\Omega)$, we know that $H_{f_1}$ is bounded on $H^2(\Omega)$. Using a consequence of the Kohn's solution of $\overline{\partial}_b$-Neumann problem (cf. \cite{FK72}), we see that $H_\phi$ is compact for  $\phi\in C^\infty(\partial\Omega)$. Since any $\phi\in C(\partial\Omega)$ can be approximated uniformly by smooth ones, it follows from Lemma 1 in \cite{SY78} that $$H_\phi=(I-P)M_\phi: H^2(\partial\Omega)\rightarrow L^2(\partial\Omega)$$ is compact for $\phi\in  C(\partial\Omega)$, where $ H^2(\partial\Omega)$ denotes   the closure of the $C^\infty$-functions on $\partial\Omega$ which can be extended holomorphically to $\Omega$. This means that $H_{f_1}$ is also compact on $H^2(\Omega)$.

Now consider the function $f_2$.  In the proof of Theorem 1.1, we have  shown  that there exists a positive constant $C$ such that
$$\|H_{f_2}g\|_{L^2(\partial\Omega)}\le  C\biggl(\| |g\cdot \overline{\partial} \widetilde{f_2}|dV\|_{W^{1/2}}+\biggl\| \frac{1}{\sqrt{-\rho}}|g\cdot \overline{\partial}\widetilde{f_2}\wedge\overline{\partial}\rho|dV\biggr\|_{W^{1/2}}\biggr)$$ for any $g\in H^2(\Omega)$. Let $\{g_j\}$ be a bounded sequence in $H^2(\Omega)$ and converge to $0$ uniformly on any compact subset of $\Omega$ as $j\to\infty$. Since   $|\overline{\partial}\widetilde{f_2}|dV$ and $\frac{1}{\sqrt{-\rho}}|\overline{\partial}\widetilde{f_2}\wedge\overline{\partial}\rho|dV$ are vanishing Carleson measures, we infer from  (2.1) and (2.2) that $$\| |g_j\cdot \overline{\partial} \widetilde{f_2}|dV\|_{W^{1/2}}\rightarrow0$$ and $$\biggl\| \frac{1}{\sqrt{-\rho}}|g_j\cdot \overline{\partial}\widetilde{f_2}\wedge\overline{\partial}\rho|dV\biggr\|_{W^{1/2}}\rightarrow0$$ as $j\rightarrow\infty$.
This yields that  $\lim\limits_{j\to\infty}\|H_{f_2}g_j\|_{L^2(\partial\Omega)}=0$. Thus $H_{f_2}$ is compact on  $H^2(\Omega)$. So we complete the proof.
\end{proof}

In \cite{Var77}, Varopoulos found that $f\in BMO(\partial\Omega)$ admits a decomposition $f=f_1+f_2$, such that $f_1\in L^\infty(\partial\Omega)$ and $f_2$ has an extension $\widetilde{f_2}$ in $\Omega$, which satisfying  that
$\lim\limits_{\lambda\rightarrow 0}\widetilde{f_2}(\zeta+\lambda \nu_0(\zeta))=\widetilde{f_2}(\zeta)$ for any $\zeta\in \partial\Omega$
and $|D\widetilde{f_2}|dV$ is a Carleson measure. This follows from a well-known result of  Carleson, i.e. if $f\in BMO(\partial \Omega)$ then there exists some Carleson measure  $d\mu$ in $\Omega$ such that $$f(\zeta)-\int_{\Omega}P(z, \zeta)d\mu(z)\in L^\infty (\partial \Omega).$$

For $f\in VMO(\partial \Omega)$ it is natural to make the following conjecture:
\medskip
\\
{\bf Conjecture 1.} {\it If $f\in VMO(\partial \Omega)$, then there exists some vanishing Carleson measure $d\mu$  in $\Omega$  such that  $$f(\zeta)-\int_{\Omega}P(z, \zeta)d\mu(z)\in C(\partial \Omega).$$ }

{\bf Remark.} Once this conjecture  is true, one may use similar arguments as Theorem 2.1.1 and Lemma 2.5.1 in \cite{Var77} to obtain the following  decomposition:
 if $f\in VMO(\partial \Omega)$, then $f=f_1+f_2$ with $f_1\in C(\partial\Omega)$ and  $\widetilde{f_2}$ being an extension of $f_2$, such that the measures $|\overline{\partial}\widetilde{f_2}|dV$ and $\frac{1}{\sqrt{-\rho}}|\overline{\partial}\widetilde{f_2}\wedge\overline{\partial}\rho|dV$ are vanishing Carleson measures in $\Omega$. As a consequence,  Theorem 3.1 implies that $H_f$ is compact from $H^2(\Omega)$ into $L^2(\partial\Omega)$ for  $f\in  VMO(\partial \Omega)$.

Now, for $0<p<\infty$ and $s>-1$, we recall that the weighted Bergman spaces $A^p_s(\Omega)$ consist of all holomorphic functions $f$ on $\Omega$ satisfying $$\|f\|_{p,s}=\biggl(\int_\Omega|f|^p(-\rho)^sdV\biggr)^{1/p}<\infty.$$
For $s=-1$, we define $A^p_{-1}(\Omega)$ to be the Hardy space $H^p(\Omega)$. Let $\Omega$  be a strongly pseudoconvex domain in $\mathbb{C}^n$, Theorem 5.4 in \cite{Bea85} gives that $\mathcal{O}(\overline{\Omega})$
is dense in  $A^p_s(\Omega)$ for $0<p<\infty$ and $s\ge -1$. This implies that $H^\infty(\Omega)$ is dense in $H^p(\Omega)$ for $0<p<\infty$. Since the projection $P$  via the Szeg\"{o} kernel $S(z, w)$ is bounded from $L^p(\partial\Omega)$  into  $H^p(\Omega)$  if  $p>1$, we may define Hankel operators $H_f$ on Hardy spaces  $H^p(\Omega)$.  On the other hand,   Theorem A and (2.2) hold for all  $p>1$. Therefore,  we  might use the same proof to obtain the following.

\begin{theorem}\label{thm 3.2}
Let  $\Omega$ be a  bounded  strongly pseudoconvex domain with smooth boundary in $\mathbb{C}^n$. Assume $f\in L^p(\partial\Omega)$ and  $p>1$. If $f\in BMO(\partial\Omega)$, then the Hankel operator $H_f$ is bounded from $H^p(\Omega)$ into $L^p(\partial\Omega)$.
\end{theorem}

Finally, we try to give a necessary condition for Hankel operator $H_f$ to be bounded on the Hardy space $H^2(\Omega)$, which shall build a bridge to solve the boundedness conjecture pointed out in Section 5.

\begin{theorem}\label{thm 3.3}
Let  $\Omega$ be a  smoothly bounded  strongly pseudoconvex domain.  If the Hankel operator $H_f$ is bounded from $H^2(\Omega)$ into $L^2(\partial\Omega)$, then the following inequality
$$ \inf\limits_{h\in H(B(w, r))}\frac{1}{\sigma(B(w, r))}\int_{B(w, r)}|f(\zeta)-h(\zeta)|^2d\sigma(\zeta)<\infty$$
holds for any $w\in \partial\Omega$ and $r>0$.
\end{theorem}

\begin{proof}

 Let $\Phi(z, w)$ be the function introduced in Section 2 and $$v(z, w)=\Phi(z, w)-\rho(z).$$ It follows that $|v(z, w)|\thicksim d(z)+d(w)+d(z', w')$.  Set $$g_z(\zeta)=\frac{d(z)^{n/2}}{v(z, \zeta)^n}, \ \ \ \ \  z\in \Omega, \ \zeta\in \overline{\Omega}.$$ Since $v(z, \zeta)$ is holomorphic in $\zeta$ for fixed $z$, and the well-known estimate
$$\int_{\partial\Omega}\frac{d\sigma(\zeta)}{|v(z, \zeta)|^{n+\alpha}}\lesssim \frac{1}{d(z)^\alpha}$$
for $\alpha>0$, we have $\sup\limits_{z\in \Omega}\|g_z\|_{H^2}\le C$ for some constant $C>0$. If $H_f$ is bounded, then $\sup\limits_{z\in \Omega}\|H_f (g_z)\|_{L^2(\partial\Omega)}\le C\|H_f\|$. It follows that
{\setlength\arraycolsep{2pt}
\begin{eqnarray*}\infty &>&\sup\limits_{z\in \Omega}\int_{\partial\Omega}|fg_z-P(fg_z)|^2d\sigma
\\ &=& \sup\limits_{z\in \Omega}\int_{\partial\Omega}\biggl|f(\zeta)-\frac{1}{g_z(\zeta)}P(fg_z)(\zeta)\biggr|^2|g_z(\zeta)|^2d\sigma(\zeta).
\end{eqnarray*}}

Let $w\in \partial\Omega$ and $B(w, r)$ be a ball on $\partial\Omega$ in the Koranyi distance. Then for $\zeta\in B(w, r)$ and $\widetilde{w}=w-rv^{(1)}$, where $\{v^{(1)}, \cdots, v^{(n)}\}$ is the orthonormal basis at $w$ with $v^{(1)}$ transversal to the boundary, we see that $$|g_{\widetilde{w}}(\zeta)|^2=\frac{d(\widetilde{w})^{n}}{|v(\widetilde{w}, \zeta)|^{2n}}\gtrsim \frac{1}{\sigma(B(w, r))}.$$
Since $\frac{1}{g_{\widetilde{w}}(\zeta)}P(fg_{\widetilde{w}})(\zeta)$ is holomorphic on $\partial\Omega$, we may choose a  holomorphic function $h$ on $B(w, r)$ satisfying
$$\frac{1}{\sigma(B(w, r))}\int_{B(w, r)}|f(\zeta)-h(\zeta)|^2d\sigma(\zeta)$$
 $$\lesssim\int_{\partial\Omega}\biggl|f(\zeta)-\frac{1}{g_{\widetilde{w}}(\zeta)}P(fg_{\widetilde{w}})(\zeta)\biggr|^2|g_{\widetilde{w}}(\zeta)|^2d\sigma(\zeta).$$
As a consequence, we obtain the desired inequality.
\end{proof}

\section{Hankel operators on  Hardy spaces of general pseudoconvex domains} \label{Sec4}

Let $\Omega$ be a bounded domain with $C^2$ boundary in $\mathbb C^n$. A positive measure $d\mu$ on $\Omega$ is called a Carleson measure if
there exists a constant $C>0$ such that
$$
\int_\Omega |g|^2 d\mu \le C \int_{\partial \Omega} |g|^2 d\sigma,\ \ \ \forall\,g\in H^2(\Omega).
$$ The starting point is the following result due to Berndtsson:
\medskip
\\
{\bf Theorem D} [cf. \cite{Berndtsson01}, see also \cite{ChenFu11}]\label{th:Berndtsson}.
{\it Let $\Omega$ be a bounded pseudoconvex domain with $C^2$ boundary in $\mathbb C^n$ and let $\rho$ be a negative $C^2$ function on $\Omega$. Suppose there exists a $C^2$ psh function $\psi$ on $\Omega$ which is continuous on $\overline{\Omega}$ and satisfies
\begin{equation}\label{eq:curvature}
\Theta:=(-\rho)\partial\bar{\partial}\psi +\partial\bar{\partial}\rho>0.
\end{equation}
Let $u_0$ be the $L^2_\psi(\Omega)-$minimal solution of the equation $\bar{\partial}u=v$. Then
$$
\int_{\partial \Omega} |u_0|^2e^{-\psi}d\sigma/|\nabla \rho| \le \int_\Omega |v|^2_\Theta e^{-\psi} dV.
$$}

Applying this result, we shall  give the following characterization for the boundedness of Hankel operators.

\begin{theorem}\label{th:general}
Let $\Omega$ be a bounded pseudoconvex domain with $C^2$ boundary in $\mathbb C^n$ and let $f\in L^2(\partial \Omega)$. Suppose the following conditions hold:
\begin{enumerate}
\item there exist  a negative $C^2$ function $\rho$ on $\Omega$ and  a $C^2$ psh function $\psi$ on $\Omega$ which is continuous on $\overline{\Omega}$ such that $(\ref{eq:curvature})$ holds.
\item there is a decomposition $f=f_1+f_2$, where $f_1\in L^\infty(\partial \Omega)$ and $f_2$ admits a Stokes extension $\tilde{f}_2$ to $\Omega$ such that $|\bar{\partial} \tilde{f}_2|^2_{\Theta} dV$ is a Carleson measure on $\Omega$.
\end{enumerate}
Then $H_f:H^2(\Omega)\rightarrow L^2(\partial \Omega)$ is bounded.
\end{theorem}

\begin{proof}
Clearly  $H_{f_1}:H^2(\Omega)\rightarrow L^2(\partial \Omega)$ is bounded. For given $g\in H^2(\Omega)$, $H_{f_2} g$ is nothing but the $L^2(\partial \Omega)-$minimal solution of the equation
$$
\bar{\partial}_b u=\bar{\partial}(\tilde{f}_2g)=g\cdot\bar{\partial}\tilde{f}_2.
$$
 It follows from Theorem D that
\begin{eqnarray*}
\int_{\partial \Omega} |H_f g|^2 d\sigma/|\nabla \rho| & \le & \int_{\partial \Omega} |u_0|^2 d\sigma/|\nabla \rho| = \lim_{r\rightarrow 1-} (1-r)\int_\Omega |u_0|^2 (-\rho)^{-r} dV\\
& \lesssim &   \lim_{r\rightarrow 1-} (1-r)\int_\Omega |u_0|^2 (-\rho)^{-r} e^{-\psi}dV\\
& \lesssim & \int_{\partial \Omega} |u_0|^2e^{-\psi}d\sigma/|\nabla \rho|\\
&\lesssim &  \int_\Omega |g|^2 |\bar{\partial} \tilde{f}_2|^2_\Theta e^{-\psi} dV\\
& \lesssim &  \int_\Omega |g|^2  |\bar{\partial}\tilde{f}_2|^2_\Theta   dV
\lesssim \int_{\partial \Omega} |g|^2 d\sigma,
\end{eqnarray*}
since $  |\bar{\partial}\tilde{f}_2|^2_\Theta dV$ is a Carleson measure. The proof is complete.
\end{proof}

Let $\Omega$ be a bounded strongly pseudoconvex domain in $\mathbb C^n$ and let $\rho\ge -1/e^2$ be a definition function which is strictly psh on $\overline{\Omega}$.  Given $\varepsilon>0$, let
$$
\psi=-[-\log(-\rho)]^{-\varepsilon}.
$$
A straightforward calculation shows that if $\varepsilon\le 1/2$ then
\begin{eqnarray*}
\partial\bar{\partial} \psi & = & \varepsilon [-\log(-\rho)]^{-\varepsilon-1} \partial\bar{\partial} [-\log(-\rho)]\\
&& -\varepsilon(1+\varepsilon) [-\log(-\rho)]^{-\varepsilon-2}\partial \log(-\rho)\wedge \bar{\partial}\log(-\rho)\\
& \ge &  \frac{\varepsilon}4 [-\log(-\rho)]^{-\varepsilon-1} \partial\bar{\partial} [-\log(-\rho)]\\
& \ge &  \frac{\varepsilon}4 [-\log(-\rho)]^{-\varepsilon-1} \rho^{-2}\partial\rho\wedge \bar{\partial} \rho,
\end{eqnarray*}
so that
\begin{eqnarray*}
\Theta & = & (-\rho)\partial\bar{\partial}\psi+\partial\bar{\partial}\rho\\
&\ge& C\left[\partial\bar{\partial}\rho+ \frac{\partial\rho\wedge \bar{\partial}\rho}{(-\rho) [-\log(-\rho)]^{1+\varepsilon}}\right].
\end{eqnarray*}
Set $N=|\bar{\partial} \rho|^{-1}\sum_j \frac{\partial \rho}{\partial \bar{z}_j}\cdot \frac{\partial}{\partial\bar{z}_j}$. For any $(0,1)-$form $\omega$ on $\Omega$, we set
$$
\omega_N=\langle \omega,N\rangle \cdot \frac{\bar{\partial}\rho}{|\bar{\partial}\rho|}\ \ \ \text{and}\ \ \ \omega_T=\omega-\omega_N.
$$
It follows that $|\bar{\partial} \tilde{f}_2|^2_\Theta dV$ is a Carleson measure provided that
$
|(\bar{\partial} \tilde{f}_2)_T|^2 dV
$
 and
$$
(-\rho) [-\log(-\rho)]^{1+\varepsilon}  |(\bar{\partial} \tilde{f}_2)_N|^2 dV
$$
are Carleson measures. Consequently, we obtain

\begin{corollary}
Let $\Omega$ be a bounded strongly pseudoconvex domain in $\mathbb C^n$ and $\rho$ a strictly psh defining function on $\Omega$. Let $f\in L^2(\partial \Omega)$. Suppose $f=f_1+f_2$, where $f_1\in L^\infty(\partial \Omega)$ and $f_2$ admits a Stokes extension $\tilde{f}_2$ such that
$
|(\bar{\partial} \tilde{f}_2)_T|^2 dV
$
 and
$$
(-\rho) [-\log(-\rho)]^{1+\varepsilon}  |(\bar{\partial} \tilde{f}_2)_N|^2 dV
$$
are Carleson measures for some $\varepsilon>0$. Then $H_f: H^2(\Omega)\rightarrow L^2(\partial \Omega)$ is bounded.
\end{corollary}

Similar ideas work for more general domains, e.g. pseudoconvex domains of finite D'Angelo type (compare \cite{ChenFu11}, Proposition 5.2).

In the  setting of strongly pseudoconvex domains, the theory of the $\overline{\partial}$-operator has been used to investigate Hankel operators $H_f$  on the Bergman space $A^2(\Omega)$. The boundedness, compactness, essential norms and Schatten ideals of $H_f$ have been completely characterized (see  [20-22] and \cite{ASS00}). In the proofs of these results,  an essential tool is the integral representation of a solution to the $\overline{\partial}$-equation.

To conclude  this section, we shall apply the $L^2$-estimates of $\overline{\partial}$-equation due to  Donnelly-Fefferman \cite{DF83} to give a new proof for  the boundedness of $H_f$  on  $A^2(\Omega)$.

Let $\Omega$ be a bounded pseudoconvex domian  and let $\psi$ be a $C^2$ plurisubharmonic function on $\Omega$.  For any $(0, 1)$ form $f$, let $$|f|^2_{\partial\overline{\partial}\psi}:=\sum\limits_{i, j=1}^n \psi^{i \, \overline{j}} f_i \overline{f_j},$$
where $(\psi^{i\, \overline{j}})=(\psi_{i \,\overline{j}})^{-1}$. According to an estimate of Donnelly and  Fefferman, one can  solve $\overline{\partial}u=f$ with the following estimate
$$\int_{\Omega}|u|^2e^{-\varphi}dV\le C_0\int_\Omega|f|^2_{\partial\overline{\partial}\psi}e^{-\varphi}dV$$
for some numerical constant $C_0$, whenever $\psi$ satisfies $\partial\overline{\partial}\psi\ge \partial \psi\wedge \overline{\partial}\psi$  and $\varphi$ is plurisubharmonic.

We first consider the case of the unit ball.
\begin{theorem}\label{thm 4.3}  Let $f\in L^2(B_n)\cap C^1(B_n)$ and  $\rho=|z|^2-1$. If $$|\rho\cdot \overline{\partial} f|+|\rho|^{1/2}|\overline{\partial} f\wedge \overline{\partial}\rho| \le C$$ for a constant $C>0$.
Then the Hankel operator $H_f=(I-P_B)M_f$ is bounded from $A^2(B_n)$ into $L^2(B_n)$, here $P_B$ is the Bergman projection.
\end{theorem}

\begin{proof}
Take $\psi=-\log(-\rho)$. A calculation shows that
$$\partial\overline{\partial}\psi =- \frac{\partial\overline{\partial}\rho}{\rho}+\frac{\partial\rho\wedge\overline{\partial}\rho}{\rho^2}\ge \partial\psi\wedge \bar{\partial}\psi$$
and $$ \partial\overline{\partial}\psi=\frac{\frac{1}{|z|^2}\sum\limits_{i, j=1}^n\overline{z_i}z_jdz_i\wedge d\overline{z_j}}{(1-|z|^2)^2}
+\frac{\sum\limits_{i=1}^ndz_i\wedge d\overline{z_j}-\frac{1}{|z|^2}\sum\limits_{i, j=1}^n\overline{z_i}z_jdz_i\wedge d\overline{z_j}}{1-|z|^2}.$$
Write $\partial\overline{\partial}\psi=\sum\psi_{i \, \overline{j}}dz_i\wedge d\overline{z_j}$.  Let $A=\frac{1}{|z|^2}(\overline{z_i}z_j)$ and   $B=I-A$, where $I$ is the identity matrix. It is easy to check  $A^2=A$ and $AB=0$. Thus we get
{\setlength\arraycolsep{2pt}
\begin{eqnarray*}\sum\psi^{i \, \overline{j}}dz_i\wedge d\overline{z_j}&=&(1-|z|^2)^2\frac{1}{|z|^2}\sum\limits_{i, j=1}^n\overline{z_i}z_jdz_i\wedge d\overline{z_j}
\\ && +
(1-|z|^2)\biggl(\sum\limits_{i=1}^ndz_i\wedge d\overline{z_j}-\frac{1}{|z|^2}\sum\limits_{i, j=1}^n\overline{z_i}z_jdz_i\wedge d\overline{z_j}\biggr).
\end{eqnarray*}}Now, for $\overline{\partial} u=\sum\limits_{i=1}^n \frac{\partial u}{\partial\overline{z_i}}d\overline{z_i}$, we compute that
{\setlength\arraycolsep{2pt}
\begin{eqnarray*}
|\overline{\partial} u|^2_{\partial\overline{\partial}\psi}&=&(1-|z|^2)^2\frac{1}{|z|^2}\sum\limits_{i, j=1}^n\overline{z_i}z_j\frac{\partial u}{\partial\overline{z_i}}\overline{\frac{\partial u}{\partial\overline{z_j}}}
\\ && +
(1-|z|^2)\biggl(\sum\limits_{i=1}^n\biggl|\frac{\partial u}{\partial\overline{z_i}}\biggr|^2-\frac{1}{|z|^2}\sum\limits_{i, j=1}^n\overline{z_i}z_j\frac{\partial u}{\partial\overline{z_i}}\overline{\frac{\partial u}{\partial\overline{z_j}}}\biggr)
 \\ &=& -(1-|z|^2)\sum\limits_{i, j=1}^n\overline{z_i}z_j\frac{\partial u}{\partial\overline{z_i}}\overline{\frac{\partial u}{\partial\overline{z_j}}}
+(1-|z|^2)|z|^2 \sum\limits_{i=1}^n\biggl|\frac{\partial u}{\partial\overline{z_i}}\biggr|^2
\\ && + (1-|z|^2)^2\sum\limits_{i=1}^n\biggl|\frac{\partial u}{\partial\overline{z_i}}\biggr|^2
\\ &=& (1-|z|^2)^2\sum\limits_{i=1}^n\biggl|\frac{\partial u}{\partial\overline{z_i}}\biggr|^2 + (1-|z|^2)\biggl(|z|^2 \sum\limits_{i=1}^n\biggl|\frac{\partial u}{\partial\overline{z_i}}\biggr|^2-\sum\limits_{i, j=1}^n\overline{z_i}z_j\frac{\partial u}{\partial\overline{z_i}}\overline{\frac{\partial u}{\partial\overline{z_j}}}\biggr)
\\ & =&(1-|z|^2)^2|\overline{\partial} u|^2+(1-|z|^2)\biggl|\overline{\partial} u\wedge\overline{\partial}|z|^2\biggr|^2
\\ & =&\rho^2|\overline{\partial} u|^2+|\rho|\cdot|\overline{\partial} u\wedge\overline{\partial}\rho|^2.
\end{eqnarray*}}Hence, for any solution $u$ orthogonal to $A^2(B_n)$,  the estimate of Donnelly and Fefferman yields that
{\setlength\arraycolsep{2pt}
\begin{eqnarray*}
\int_{B_n}|u|^2dV&\le & C_0\int_{B_n}|\overline{\partial} u|^2_{\partial\overline{\partial}\psi}dV
\\ &=& C_0 \int_{B_n}\biggl(\rho^2|\overline{\partial} u|^2+|\rho|\cdot|\overline{\partial} u\wedge\overline{\partial}\rho|^2\biggr)dV.
\end{eqnarray*}}

Next, for any $g\in A^2(B_n)$, consider the equation
$$\overline{\partial} u =\overline{\partial} (fg)=g\cdot \overline{\partial} f.$$
It is obvious that $g\cdot \overline{\partial} f$ is $\overline{\partial}$-closed and $u=(I-P)(fg)=H_f g$ is the minimal solution to this equation. Thus,  the above discussion gives that
{\setlength\arraycolsep{2pt}
\begin{eqnarray*}
\int_{B_n}|H_f g|^2dV&\le &C_0\int_{B_n}|g\cdot \overline{\partial} f |^2_{\partial\overline{\partial}\psi}dV
\\ &=&C_0 \int_{B_n}|g|^2\biggl(|\rho\cdot \overline{\partial} f|^2+|\rho|\cdot|\overline{\partial} f\wedge\overline{\partial}\rho|^2\biggr)dV
\\ &\le &C_0C^2 \int_{B_n}|g|^2dV,
\end{eqnarray*}}since  $|\rho\cdot \overline{\partial} f|+|\rho|^{1/2}|\overline{\partial} f\wedge \overline{\partial}\rho| \le C$.
It follows that $H_f$  is bounded from $A^2(B_n)$ into $L^2(B_n)$. Moreover, we obtain the following norm estimate
$$\|H_f\|\le C_0\sup\limits_{z\in B_n}\biggl(|\rho\cdot \overline{\partial} f|^2+|\rho|\cdot|\overline{\partial} f\wedge\overline{\partial}\rho|^2\biggr)^{1/2}.$$
So the proof is complete.
\end{proof}

As pointed out on page 89 of \cite{AC00}, if $f$ is a $(0, 1)$ form defined in the strongly pseudoconvex domain $\Omega=\{\rho<0\}$, we also have
$$|f|^2_{\partial\overline{\partial}\log(1/-\rho)}=\frac{1}{B}\biggl(\rho^2|f|^2_\beta+|\rho| \cdot|f\wedge \overline{\partial}\rho|^2_\beta\biggr),$$
where $B=-\rho+|\partial\rho|_\beta\thicksim 1$ and $|\cdot |_\beta$ denotes the norm induced by the metric form $\beta=\frac{i}{2}\partial\overline{\partial}\rho$, which is equivalent to the Euclidean metric, since $\rho$ is strictly plurisubharmonic. Therefore, using  similar ideas as  Theorem 4.3, we get the following result.

\begin{theorem}\label{thm 4.4}  Let $\Omega$ be a strongly pseudoconvex domain with smooth boundary and $f\in L^2(\Omega)\cap C^1(\Omega)$. If $$|\rho\cdot \overline{\partial} f|+|\rho|^{1/2}|\overline{\partial} f\wedge \overline{\partial}\rho| \le C$$ for a constant $C>0$.
Then the Hankel operator $H_f=(I-P_B)M_f$ is bounded from $A^2(\Omega)$ into $L^2(\Omega)$.
\end{theorem}

This result  is just Theorem 3.5 of \cite{Li94}, where the authors also showed  that this condition is necessary.  Now, we once again find  the relationship between  Hankel operators and $\overline{\partial}$-theory. Moreover,  using the inequality
$$\|H_f g\|^2_{A^2(\Omega)}\lesssim\int_{\Omega}|g|^2\biggl(|\rho\cdot \overline{\partial} f|^2+|\rho|\cdot|\overline{\partial} f\wedge\overline{\partial}\rho|^2\biggr)dV,$$
 we may easily deduce the compactness of $H_f$ on $A^2(\Omega)$ as Theorem 3.6 in \cite{Li94}.

\section{Proof of Theorem 1.2}\label{Sec5}

$(1)\Rightarrow (2)$. For any $z\in D$, the normalized reproducing kernel $k_z\in H^2(D)$ and $\|k_z\|_{H^2(D)}=1$.  So we have
$$\frac{1}{2\pi}\int |f k_z-P(f k_z)|^2d\theta=\|H_f(k_z)\|^2_{H^2(D)}\le C_1$$
for a positive constant $C_1$. It is easy to see that $\frac{1}{k_z} P(f k_z)\in H^2(D)$ and $$P_z(\theta)=\frac{1-|z|^2}{|e^{i\theta}-z|^2}=|k_z|^2.$$ Thus  we get
{\setlength\arraycolsep{2pt}
\begin{eqnarray*}
&& \inf\limits_{h\in H^2}\frac{1}{2\pi}\int|f-h|^2P_z(\theta)d\theta  \le  \frac{1}{2\pi}\int |f -\frac{1}{k_z}P(f k_z)|^2|k_z|^2d\theta
\\ && =\frac{1}{2\pi}\int |f k_z-P(f k_z)|^2d\theta\le C_1,
\end{eqnarray*}}which implies (2).

$(2)\Rightarrow (3)$. There exists a constant $M>0$, such that $F(z)\le M$ for any $z\in D$. For  $z=0$, we have
$$F(0)= \inf\limits_{h\in H^2}\frac{1}{2\pi}\int|f-h|^2d\theta\le M.$$
Choose $h_0 \in H^2(D)$ which satisfies $$\frac{1}{2\pi}\int|f-h_0|^2d\theta=\inf\limits_{h\in H^2}\frac{1}{2\pi}\int|f-h|^2d\theta.$$
It is clear that $h_0$ is the orthogonal projection of $f$ onto $H^2(D)$ and $f-h_0$ is conjugate analytic. Moreover, we have $f(0)-h_0(0)=0$. Hence,
$$|\nabla (f-h_0)|^2=2|\partial(f-h_0)/\partial \overline{z}|^2=2|\partial f/\partial \overline{z}|^2.$$
Combining this with  the Littlewood-Paley indetity (see, p. 230), we obtain
$$\frac{1}{2\pi}\int|f-h_0|^2d\theta = \frac{1}{\pi}\iint\limits_D|\nabla (f-h_0)|^2\log\frac{1}{|z|}dxdy$$
$$= \frac{2}{\pi}\iint\limits_D\biggl|\frac{\partial f}{\partial \overline{z}}\biggr|^2\log\frac{1}{|z|}dxdy\le M.\eqno(5.1)$$

For fixed  $z_0\in D$, let $h_1$ be chosen to attain the infimum in (1.1) with respect $P_{z_0}(\theta)d\theta$. Applying similar arguments as in the proof of Theorem 3.5 on page 372 of \cite{Gar81}, we see   that $f-h_1$ is conjugate analytic and $f(z_0)-h_1(z_0)=0$, because in the Hilbert space $L^2(P_{z_0}(\theta)d\theta)$, the function $h_1$ is the orthogonal projection of $f$ onto $H^2(D)$. This means
$$|\nabla (f-h_1)|^2=2|\partial(f-h_1)/\partial \overline{z}|^2=2|\partial f/\partial \overline{z}|^2.$$
Thus the Littlewood-Paley indetity gives
$$\frac{1}{2\pi}\int|f-h_1|^2P_{z_0}(\theta)d\theta =\frac{2}{\pi}\iint\limits_D\biggl|\frac{\partial f}{\partial \overline{z}}\biggr|^2\log\biggl|\frac{1-\overline{z_0}z}{z-z_0}\biggr|dxdy\le M.\eqno(5.2) $$
Notice that
$$(1-|z|^2)\frac{1-|z_0|^2}{|1-\overline{z_0}z|^2}=1-\biggl|\frac{z-z_0}{1-\overline{z_0}z}\biggr|^2\le 2\log\biggl|\frac{1-\overline{z_0}z}{z-z_0}\biggr|.$$
Combining this inequality  with (5.2), we have a constant $C_2>0$ such that
{\setlength\arraycolsep{2pt}
\begin{eqnarray*}
&&\iint\limits_{|z|>1/4} \frac{1-|z_0|^2}{|1-\overline{z_0}z|^2}\cdot\biggl|\frac{\partial f}{\partial \overline{z}}\biggr|^2\log\frac{1}{|z|}dxdy
\\ && \le C_2\iint\limits_{|z|>1/4} \biggl|\frac{\partial f}{\partial \overline{z}}\biggr|^2(1-|z|^2)\frac{1-|z_0|^2}{|1-\overline{z_0}z|^2}dxdy
\\ && \le 2C_2\iint\limits_D \biggl|\frac{\partial f}{\partial \overline{z}}\biggr|^2\log\biggl|\frac{1-\overline{z_0}z}{z-z_0}\biggr|dxdy\le C_2M\pi.
\end{eqnarray*}}On the other hand, we infer from   (5.1) that
{\setlength\arraycolsep{2pt}
\begin{eqnarray*}
&&\iint\limits_{|z|\le 1/4} \frac{1-|z_0|^2}{|1-\overline{z_0}z|^2}\cdot\biggl|\frac{\partial f}{\partial \overline{z}}\biggr|^2\log\frac{1}{|z|}dxdy
\\ && \le C_3\iint\limits_{|z|\le 1/4} \biggl|\frac{\partial f}{\partial \overline{z}}\biggr|^2\log\frac{1}{|z|}dxdy
\le C_3M\pi.
\end{eqnarray*}}for a constant $C_3>0$. All these arguments yield that
$$\sup\limits_{z_0\in D}\iint\limits_D \frac{1-|z_0|^2}{|1-\overline{z_0}z|^2}\cdot\biggl|\frac{\partial f}{\partial \overline{z}}\biggr|^2\log\frac{1}{|z|}dxdy\le C$$
for   a positive constant $C$. In view of  Lemma 3.3 in Chapter VI of  \cite{Gar81}, we see that $|\frac{\partial f}{\partial \overline{z}}|^2\log\frac{1}{|z|}dxdy$ is a Carleson measure.

$(3)\Rightarrow (1)$. For any $g\in H^2(D)$, we see that $fg-P(fg)$ is conjugate analytic and $f(0)g(0)-P(fg)(0)=0$, because $(H^2(D))^\perp=\overline{H^2(D)}$. If $|\frac{\partial f}{\partial \overline{z}}|^2\log\frac{1}{|z|}dxdy$ is a Carleson measure, then
 the Littlewood-Paley identity implies
{\setlength\arraycolsep{2pt}
\begin{eqnarray*}
\|H_f(g)\|^2_{H^2(D)}&=& \frac{1}{2\pi}\int|fg-P(fg)|^2d\theta
\\ &=& \frac{1}{\pi}\iint\limits_D \biggl|\nabla(fg-P(fg))\biggr|^2\log\frac{1}{|z|}dxdy
\\ &=& \frac{2}{\pi}\iint\limits_D |g|^2\cdot\biggl|\frac{\partial f}{\partial \overline{z}}\biggr|^2\log\frac{1}{|z|}dxdy
\lesssim \|g\|^2_{H^2(D)}.
\end{eqnarray*}}Thus, $H_f$ is bounded on $H^2(D)$. This completes the proof of Theorem 1.2. \qquad \qquad  $\Box$

\medskip

{\bf Remark.} Recall that $f\in BMO(\partial D)$ if and only if
$$d\lambda_f=|\nabla f(z)|^2 \log\frac{1}{|z|}dxdy$$
is a Carleson measure by Theorem 3.4 of \cite{Gar81}. Let $$d\nu_f=\biggl|\frac{\partial f}{\partial \overline{z}}\biggr|^2\log\frac{1}{|z|}dxdy.$$
Since $|\nabla f(z)|^2=2(|\partial f/\partial \overline{z}|^2+|\partial f/\partial z|^2)$ and $\overline{\partial f/\partial z}=\partial \overline{f}/\partial \overline{z}$, we have $\lambda_f=2(\nu_f+\nu_{\overline{f}})$. In particular, when $f$ is real-valued, $\lambda_f$
and $\nu_f$ are equivalent. Thus, Theorem 1.2 implies a well-known result, that is, $H_f$ is bounded on $H^2(D)$ if and only if $f\in BMO(\partial D)$ for real-valued $f$.

On the other hand, observing Theorem 1.2 and  the corresponding result on the Bergman space $A^2(\Omega)$, we have the following conjecture for the boundedness of  Hankel operators on  $H^2(\Omega)$.
\medskip
\\
{\bf Conjecture 2.}
{\it Suppose that $\Omega$ is a bounded  strongly pseudoconvex domain with smooth boundary and $f\in L^2(\partial\Omega)$. Then the Hankel operator $H_f$ is bounded from $H^2(\Omega)$ into $L^2(\partial\Omega)$,
if and only if  $f$  admits a decomposition $f=f_1+f_2$ with $f_1\in L^\infty(\partial\Omega)$ and $f_2$ having  an extension   $\widetilde{f_2}$ in $\Omega$, which satisfies that
 $|\overline{\partial}\widetilde{f_2}|dV$ and $\frac{1}{\sqrt{-\rho}}|\overline{\partial}\widetilde{f_2}\wedge\overline{\partial}\rho|dV$ are  Carleson measures in $\Omega$.}

\section{$H_f$ and $H_{\overline{f}}$ on Hardy spaces of strongly pseudoconvex domains}\label{Sec6}

For the unit disc $D$,  it is well known that for  $f\in L^2(\partial D)$,  both $H_f$ and $H_{\overline{f}}$ are bounded on  $H^2(D)$ if and only if $f\in BMO(\partial D)$. This result can be generalized to  the unit ball  \cite{Xia04} and  the key tool is the following inequality:
$$\|H_f k_z\|^2_2+\|H_{\overline{f}} k_z\|^2_2\ge \|f k_z\|^2_2+|\langle f k_z, k_z\rangle|^2,$$
where $k_z$ is the normalized reproducing kernel, which  can be proved by using  automorphisms of the ball. In this section, we denote the norm  $\|\cdot\|_2:=\|\cdot\|_{H^2}$.

Let us consider the corresponding result on the  strongly pseudoconvex domain $\Omega$. Since there might be  no nontrivial holomorphic automorphisms for general strongly pseudoconvex domains in $\mathbb{C}^n$, we borrow some ideas from  \cite{BL95} to obtain the following.

\begin{theorem}\label{thm 6.1}  Let $\Omega$ be a strongly pseudoconvex domain with smooth boundary and   $f\in L^2(\partial\Omega)$. Suppose that for each fixed $z\in \Omega$, $1/S(\cdot, z)$ is holomorphic and bounded in $\Omega$.
Then the Hankel operators both $H_f$  and $H_{\overline{f}}$ are bounded from $H^2(\Omega)$ into $L^2(\partial\Omega)$ if and only if $f$ lies in $BMO(\partial\Omega)$.
\end{theorem}

\begin{proof}
Let $S(\zeta, z)$ be the Szeg\"{o} kernel and $k_z(\zeta)=\frac{S(\zeta, z)}{\sqrt{S(z, z)}}$
be the normalized kernel. For fixed $\zeta\in \partial\Omega$, the function $S(z, \zeta)$ is holomorphic in $\Omega$.
Let $b_z(\zeta)=\frac{S(\zeta, z)}{S(z, \zeta)}$. Then for any $z\in \Omega$, we have
$$\langle H_f(k_z), b_z \overline{H_{\overline{f}}(k_z)}\rangle=\langle H_f(k_z)H_{\overline{f}}(k_z), b_z\rangle.$$

Note  that
{\setlength\arraycolsep{2pt}
\begin{eqnarray*}
H_f(k_z)H_{\overline{f}}(k_z) &=& (I-P)(f k_z)(I-P)(\overline{f}k_z)
\\ &=& |f|^2k^2_z-fk_zP(\overline{f}k_z)-\overline{f}k_zP(fk_z)+P(fk_z)P(\overline{f}k_z).
\end{eqnarray*}}Since $k_z^2 \, \overline{b_z}=|k_z|^2$, we obtain $$\langle |f|^2k_z^2, b_z\rangle=\|f k_z\|^2_2.$$
From $\overline{k_z}\, b_z=k_z$, we see that
$$\langle fk_z P(\overline{f}k_z), b_z\rangle=\langle  P(\overline{f}k_z), \overline{fk_z}\, b_z\rangle=\langle  P(\overline{f}k_z), \overline{f}k_z\rangle=\|P( \overline{f} k_z)\|^2_2$$ and similarly
$$\langle \overline{f}k_z P(fk_z), b_z\rangle=\|P(f k_z)\|^2_2.$$

Since  $1/S(\cdot, z)$ is holomorphic and bounded in $\Omega$ for each fixed $z\in \Omega$, the function $k_z^{-1}$ is well defined.
Using the reproducing property of the Szeg\"{o} kernel, we have
{\setlength\arraycolsep{2pt}
\begin{eqnarray*}\langle P(fk_z) P(\overline{f}k_z), b_z\rangle&=&\langle P(fk_z) P(\overline{f}k_z)k_z^{-1}, k_z\rangle
\\ &=&P(fk_z)(z) P(\overline{f}k_z)(z) S(z, z)^{-1}.
 \end{eqnarray*}}Now, $ P(z, \zeta)=\frac{S(\zeta, z) S(z, \zeta)}{S(z, z)}$ is the Poisson-Szeg\"{o} kernel. Thus, we get
{\setlength\arraycolsep{2pt}
\begin{eqnarray*}
P(fk_z)&=&\sqrt{S(z,z)}\int_{\partial\Omega} f(\zeta)\frac{S(\zeta, z) S(z, \zeta)}{S(z, z)} d\sigma(\zeta)
\\ &=& \sqrt{S(z,z)}\int_{\partial\Omega}f(\zeta) P(z,\zeta) d\sigma(\zeta)=\sqrt{S(z,z)} \widetilde{f}(z),
\end{eqnarray*}}It follows that
$$\langle P(fk_z) P(\overline{f}k_z), b_z\rangle=\widetilde{f}(z)\widetilde{\overline{f}}(z)=|\widetilde{f}(z)|^2$$
because of $\widetilde{\overline{f}}=\overline{\widetilde{f}}$. Therefore, the above discussions show that
$$\langle H_f(k_z), b_z \overline{H_{\overline{f}}(k_z)}\rangle=\|f k_z\|^2_2-\|P(f k_z)\|^2_2-\|P( \overline{f} k_z)\|^2_2+|\widetilde{f}(z)|^2.$$

Recall that $\|H_f(k_z)\|^2_2=\|f k_z\|^2_{H^2}-\|P(f k_z)\|^2_2$ and $\|H_{\overline{f}}(k_z)\|^2_2=\|\overline{f} k_z\|^2_2-\|P(\overline{f} k_z)\|^2_2$. This means
$$\langle H_f(k_z), b_z \overline{H_{\overline{f}}(k_z)}\rangle=\|H_f(k_z)\|^2_2+\|H_{\overline{f}}(k_z)\|^2_2-\|\overline{f} k_z\|^2_2+|\widetilde{f}(z)|^2.$$
Applying this identity and the Schwarz inequality, we obtain
$$\|\overline{f} k_z\|^2_2-|\widetilde{f}(z)|^2\le \frac{3}{2}\biggl(\|H_f(k_z)\|^2_2+\|H_{\overline{f}}(k_z)\|^2_2\biggr).\eqno(6.1)$$

On the other hand, in Section 2, we have talked about  an  equivalent norm for $BMO$, i.e. $f\in BMO(\partial\Omega)$ if and only if
$$\sup\limits_{z\in \Omega}\int_{\partial\Omega}|f(\zeta)-\widetilde{f}(z)|^2P(z, \zeta)d\sigma(\zeta)<\infty.$$
In fact, $P(z, \zeta)=|k_z(\zeta)|^2$, so we get
{\setlength\arraycolsep{2pt}
\begin{eqnarray*}&&\int_{\partial\Omega}|f(\zeta)-\widetilde{f}(z)|^2P(z, \zeta)d\sigma(\zeta)
\\ & =&\int_{\partial\Omega}|f(\zeta)|^2P(z, \zeta)d\sigma(\zeta)- \overline{\widetilde{f}(z)}\int_{\partial\Omega}f(\zeta)P(z, \zeta)d\sigma(\zeta)
\\ && -\widetilde{f}(z)\int_{\partial\Omega}\overline{f(\zeta)}P(z, \zeta)d\sigma(\zeta)+|\widetilde{f}(z)|^2\int_{\partial\Omega}P(z, \zeta)d\sigma(\zeta)
\\ & =&\|\overline{f} k_z\|^2_2-|\widetilde{f}(z)|^2.
\end{eqnarray*}}As a consequence, if both $H_f$  and $H_{\overline{f}}$ are bounded, then  (6.1) gives
$$\int_{\partial\Omega}|f(\zeta)-\widetilde{f}(z)|^2P(z, \zeta)d\sigma(\zeta)\le  \frac{3}{2}\biggl(\|H_f(k_z)\|^2_2+\|H_{\overline{f}}(k_z)\|^2_2\biggr)<\infty.$$
This implies that $f\in BMO(\partial\Omega)$.

For another direction, Theorem 1.1 gives the desired result and the proof is complete.
\end{proof}

\end{document}